\documentclass{article}
\usepackage[english]{babel}
\usepackage{amsmath,amsfonts,amsthm,amssymb,amscd,bm,xcolor}

 \usepackage{verbatim}

\newcommand{\p}{\partial}
\newcommand{\e}{\varepsilon}

\newcommand{\R}{{\mathbb R}}

\newcommand{\T}{{\mathbb T}}
\newcommand{\N}{{\mathbb N}}

\newcommand{\al}{{\alpha}}

\newcommand{\FF}{{\cal F}}

\newcommand{\KK}{{\cal K}}

\newcommand{\RR}{{\cal R}}

\newcommand{\ph}{\varphi}
\newcommand{\Z}{ \mathbb{Z}}
\newcommand{\dd}{{\textup d}}

\newcommand{\const}{\mathop{\rm const}\nolimits}

\theoremstyle{plain}
\newtheorem{theorem}{Theorem}[section]
\newtheorem*{main theorem}{Main theorem}
\newtheorem{lemma}[theorem]{Lemma}%[section]
\newtheorem{proposition}[theorem]{Proposition}

\newtheorem{definition}[theorem]{Definition}%[section]

\theoremstyle{remark}

\newtheorem{example}[theorem]{Example}
\newcommand{\de}{\delta}
\numberwithin{equation}{section}
\begin{document}

\author{Hayk Nersisyan}
\date{}
\title{ Controllability of the 3D compressible  Euler system}

 \maketitle\begin{center}

CNRS UMR8088, D\'epartement de Math\'ematiques\\
Universit\'e de Cergy--Pontoise, Site de Saint-Martin\\
            2 avenue Adolphe Chauvin\\
       95302 Cergy--Pontoise Cedex, France\\
       E-mail: Hayk.Nersisyan@u-cergy.fr
 \end{center}

\vspace{15 pt}

 {\small\textbf{Abstract.} The paper is devoted to the
 controllability problem for 3D compressible Euler system. The control is a finite-dimensional external force acting only on the velocity equation. We show
 that the velocity and density of the fluid are simultaneously
 controllable. In particular, the system is approximately controllable and exactly controllable in projections.
  }\\\\
 \tableofcontents
 \newpage
 \section{Introduction}

The time evolution of an isentropic ideal gas is described by the
compressible Euler system
\begin{align}
\rho(\p_t{\bm{u}}+(\bm{u}\cdot\nabla)\bm{u})+ \nabla p(\rho)&=\rho\bm{f}, \label{1:E.I1}\\
\p_t\rho+\nabla\cdot(\rho \bm{u})&=0,\label{1:E.I2}\\
\bm{u}(0)=\bm{u}_0, \quad \rho(0)&=\rho_0,\label{1:E.I3}
\end{align}
where $\bm{u}= (u_1,u_2,u_3)$ and $\rho$  are unknown velocity
field and density of the gas, $p$ is the pressure and  $\bm{f}$ is
the external force, $\bm{u}_0$ and $\rho_0$ are the initial
conditions.   We assume that the space variable $\bm
x=(x_1,x_2,x_3)$ belongs to the 3D torus $\T^3=\R^3/2\pi\Z^3$.

Problem  (\ref{1:E.I1})-(\ref{1:E.I3}) can be reduced by a simple
change of variables to a quasi-linear   symmetrizable hyperbolic
system. Thus local-in-time existence   and uniqueness of a smooth
solution is well known (for instance, see \cite{Kato1, tey1}).
  Moreover, a blow-up criterion  holds for the
compressible Euler equation (see \cite[Section 16, Proposition
2.4]{tey1}).

The aim of this paper is the study of some controllability issues
for  system (\ref{1:E.I1})-(\ref{1:E.I3}). We suppose that the
external force is of the form $\bm f=\bm{\tilde f}+\bm\eta$, where
$\bm{\tilde f}$ is any given function and $\bm \eta$ is the
control taking values in a finite-dimensional space.   Let $  H^k$
be the Sobolev space of order $k$ on $\T^3$ and let $\bm H^k$   be
the space of vector functions $\bm u = (u_1, u_2, u_3)$ with
components in $H^k$. For both spaces, we denote by $\|\cdot\|_k$
the corresponding norms. We denote $J_T:=[0, T ]$.  The following
theorem is our main result.
\begin{main theorem}\label{1:T.Int}   Let $ k\ge4$ and $\bm{\tilde f}\in C^\infty(J_T,\bm{H}^{k+2})$.
 There is a
finite-dimensional space $\bm E\subset \bm H^k$    with $\dim \bm
E=45$  such that for any constants $T,\e>0$, for any continuous
function $\bm F:\bm H^k\times H^k\rightarrow \R^N$ admitting a
right  inverse, for any functions $\bm u_0, \hat {\bm u}\in \bm
H^k$ and $ \rho_0,\hat \rho\in H^k$ with
\begin{align}\label{1:E.Erho}
\int_{\T^3}\rho_0\dd\bm x=\int_{\T^3}\hat\rho\dd\bm x
\end{align}
 there is a   smooth control $\bm \eta: J_T\rightarrow \bm E$  such that system
(\ref{1:E.I1})-(\ref{1:E.I3}) has a unique  regular solution $(\bm
u, \rho)$, which  verifies
\begin{align}
&\|(\bm u(T), \rho(T)) -(\hat{\bm u}, \hat{\rho})\|_{\bm H^k\times
H^k}<\e,\nonumber\\& \bm F(\bm u(T), \rho(T))=\bm F(\hat{\bm u},
\hat{\rho}).\nonumber
\end{align}
\end{main theorem}
See Subsection \ref{1:subsect3.1} for the exact formulation. We
stress that condition (\ref{1:E.Erho}) is essential, because
integrating (\ref{1:E.I2}), we get  $\int\rho(\cdot,\bm x)\dd \bm
x=\const$.

Before turning to the ideas of the proof, let us describe in a few
words some previous results on the controllability of Euler and
Navier--Stokes systems. Li and Rao \cite {LiRao} proved a local
exact boundary controllability property for general 1D first-order
quasi-linear  hyperbolic equations. Exact boundary controllability
problems for weak entropy solutions of 1D compressible  Euler
system has been established by Glass \cite{Gla2}. Controllability
of incompressible Euler and Navier--Stokes systems has been
studied by several authors. Coron \cite {cor} introduced the
return method to show exact boundary controllability of 2D
incompressible Euler system. Glass  \cite {gla} generalized this
result for 3D Euler system.
    Exact controllability of
  Navier--Stokes systems with  control supported by a
given domain was studied by Coron and Fursikov \cite{corfurs},
Fursikov and Imanuvilov \cite{fuim}, Imanuvilov \cite{ima2001},
Fern\'{a}ndez-Cara
 et al. \cite{fpgi}. Agrachev and Sarychev \cite{agr1, agr2}
proved controllability of 2D Navier--Stokes and 2D Euler equations
with finite-dimensional external control. Rodrigues \cite{rod}
used Agrachev--Sarychev method  to prove controllability of the 2D
Navier--Stokes equation on the rectangle with Lions boundary
condition. Shirikyan  \cite{shi1, shi2} generalized this method to
the case of 3D Navier--Stokes equation. Furthermore, he shows
\cite{shi3} that 2D Euler equation is not not exactly controllable
by a finite-dimensional external force. In \cite{hn}, we show that
in the case of 3D Euler equation, the velocity and pressure are
exactly controllable in projections.

One of the main difficulties of the proof of Main theorem
 is the fact that the control $\bm \eta$ acts only on
the first equation. We combine the Agrachev--Sarychev method with
a perturbative result for compressible Euler equations and a
property of the transport equation to prove that   the  velocity
$\bm u$ and the density $\rho$  can be controlled simultaneously
with the help of a finite-dimensional external force $\bm \eta$.
   The Agrachev--Sarychev method is based on construction of an
increasing sequence of finite-dimensional spaces $\bm E_n\subset
\bm H^k, n\ge 0$ such that
\begin{enumerate}
\item[(i)] The
system is controllable with $\bm E_N$-valued controls for some
$N\ge 1$.
\item[(ii)]Controllability of the system with $\bm
\eta\in \bm E_n$ is equivalent to that  with $\bm \eta\in \bm
E_{n+1}$.
\end{enumerate}
As in the case of incompressible Euler and Navier--Stokes systems,
the proof of property $(i)$ is deduced from the hypothesis that
$\bm E_\infty:=\cup_{n=0}^\infty \bm E_n$ is dense in $\bm H^k$
and from the fact that for any functions $\bm V_0$, $\bm V_1$
there is a control (not necessarily $\bm E$-valued)  which steers
the system from  $\bm V_0$ to $\bm V_1$. As the control acts only
on the first equation, along with  (\ref{1:E.I1})-(\ref{1:E.I2})
we need to consider  the control system
\begin{align}
\rho(\p_t{\bm{u}}+((\bm{u}+\bm \xi )\cdot\nabla)(\bm{u}+\bm \xi ))+ \nabla p(\rho)&=\rho(\bm{\tilde f} +\bm \eta), \label{1:E.I11}\\
\p_t\rho+\nabla\cdot(\rho (\bm{u}+\bm \xi ))&=0.\label{1:E.I21}
\end{align}
For any $\bm V_0$ and $\bm V_1$ we find  controls  $\bm \xi, \bm
\eta $ such that the solution of (\ref{1:E.I11})-(\ref{1:E.I21})
links $\bm V_0$ and $\bm V_1$. Now to prove $(i)$, it suffices to
show that the control systems (\ref{1:E.I1})-(\ref{1:E.I2})  and
(\ref{1:E.I11})-(\ref{1:E.I21}) are equivalent. This can be done
by a simple change of the variable $\bm v=\bm{u}+\bm \xi$. To
establish property $(ii)$, we first show that the controllability
of (\ref{1:E.I1})-(\ref{1:E.I2}) with $\bm \eta\in \bm E_{n+1}$ is
equivalent to that of the system
\begin{align}
\rho(\p_t{\bm{u}}+((\bm{u}+\bm \xi )\cdot\nabla)(\bm{u}+\bm \xi ))+ \nabla p(\rho)&=\rho(\bm{\tilde f} +\bm \eta), \label{1:E.I12}\\
\p_t\rho+\nabla\cdot(\rho  \bm{u} )&=0\label{1:E.I22}
\end{align}
with $\bm \eta\in \bm E_{n }$ and $\bm \xi \in \bm E_{n}$. Here we
use the ideas from \cite{agr1, agr2, shi1,shi2,hn}. Then using a
continuity property of the resolving operator of compressible
Euler system (see Theorem \ref{1:T.te2}), we show that control
systems (\ref{1:E.I12})-(\ref{1:E.I22})  and
(\ref{1:E.I11})-(\ref{1:E.I21}) are also equivalent. We refer the
reader
to Section \ref {1:S.4.2} for a detailed proof of this property.\\

\textbf{Acknowledgments.}   The author would like to express deep
gratitude to Armen Shirikyan for drawing his attention to this
problem and  for many valuable discussions   and to the referees
for their detailed   comments and suggestions which have helped to
improve the paper.
\\
\newline

\textbf{Notation.}   We use   bold characters to denote vector
functions.  Let $X$ be a Banach  space endowed with the norm
$\|\cdot\|_X$. For $1\leq p<\infty$  let $L^p(J_T,X)$ be the space
of measurable functions $u: J_T \rightarrow X$ such that
\begin{equation}
\|u\|_{L^p(J_T,X)}:=\bigg(\int_{0}^T \|u\|_X^p\dd s
\bigg)^{\frac{1}{p}}<\infty.\nonumber
\end{equation}
   The space of continuous functions $u: J_T
\rightarrow X$ is denoted by $C(J_T,X)$. We denote by $C$ a
constant whose value may change from line to line. We write $\int
\! f(x) \dd x$ instead of   $\int_{\T^3} f(x) \dd x$.   Let
$\delta_{i,j}$ be the Kronecker delta, i.e, $\delta_{i,j}=0$ if $i
\neq j$ and $\delta_{i,i}=1$.

\section{Preliminaries on 3D  compressible Euler system}\label{1:S:1}
   \subsection{ Symmetrizable hyperbolic systems } In this
subsection, we recall some results on local existence  of
symmetrizable hyperbolic systems. Let us consider the system
  \begin{align}\label{1:E:symmmet:1}
\p_t{\bm{v}}&+\sum_{i=1}^n\bm A_i(t,\bm x, \bm v)\p_i \bm v+\bm
G(t,\bm x, \bm v)=0, &\bm v(0)= \bm v_0 .
\end{align}
We say that (\ref{1:E:symmmet:1}) is a quasi-linear  symmetric
hyperbolic system if matrices $A_i$ are symmetric, i.e., $
A_i=A_i^*$. If  functions $\bm A_i,\bm G$ are smooth and system
(\ref{1:E:symmmet:1}) is  symmetric hyperbolic, then for any $\bm
v_0\in \bm H^k$, $k> n/2 +1$ there exists $T>0$ such that system
(\ref{1:E:symmmet:1}) has a   solution $\bm v \in
C(J_{T},\bm{H}^{k})$ (see \cite{Kato1} or \cite[Chapter 16]{tey1}
for an exact statement). Now   consider a more general case:
 \begin{align}\label{1:E:symmmet:2}
\p_t{\bm{u}}&+\sum_{i=1}^n\bm B_i(t,\bm x, \bm u)\p_i \bm u+\bm
H(t,\bm x, \bm u)=0, &\bm u(0)= \bm u_0,
\end{align}
where $\bm B_i $ are such that there exists a positive definite
matrix  $\bm B_0$ such that  $\bm B_0\cdot~\!\!\bm B_i$ are
symmetric.
 These systems are called quasi-linear symmetrizable hyperbolic systems.
 As it is remarked in \cite[Chapter 16, p. 366]{tey1}, we have the
 following local well-posedness of this system.
\begin{theorem}\label{T:symeths}
Let $\bm u_0\in \bm H^k$, $k>n/2 +1$  and $\bm B_i,\bm H \in L^2
(J_T,\bm H^k\times\bm H^k)$. Then there exists $T_0>0$, which
depends on
$$\|\bm u_0\|_k+\|\bm B_i\|_{L^2(J_T,\bm H^k\times\bm
H^k)}+\|\bm H\|_{L^2 (J_T,\bm H^k\times\bm H^k)},$$ such that
system (\ref{1:E:symmmet:1}) has a unique solution $\bm u\in
C(J_{T_0},\bm{H}^{k})$.
\end{theorem}
\subsection{ Well-posedness of the Euler equations}

Let us consider the compressible Euler system
\begin{align}
\rho(\p_t{\bm{u}}+(\bm{u}\cdot\nabla)\bm{u})+ \nabla p(\rho)&=\rho\bm{f}, \nonumber\\
\p_t\rho+\nabla\cdot(\rho \bm{u})&=0,\nonumber\\
\bm{u}(0)=\bm{u}_0, \quad \rho(0)&=\rho_0.\nonumber
\end{align}
 We
study the case in which there is no vacuum,  so that  the initial
density is separated from zero.  Let us show that in this case the
above problem can be reduced to a quasi-linear symmetrizable
hyperbolic system.  Setting $g=\log\rho$ and $h(s)=p'(e^{s})$, the
above system takes the equivalent form
\begin{align}
\p_t{\bm{u}}+( \bm{u} \cdot\nabla) \bm{u}  + h(g)\nabla g&=\bm{f}, \label{1:E:Eulersym1}\\
(\p_t+ \bm{u} \cdot\nabla)g+\nabla\cdot  \bm{u} &=0,\label{1:E:Eulersym2}\\
\bm{u}(0)=\bm{u}_0, \quad g(0)&=g_0.\label{1:E:Eulersym3}
\end{align}
In what follows, we shall deal with the more general system
\begin{align}
\p_t{\bm{u}}+((\bm{u}+\bm{\zeta})\cdot\nabla)(\bm{u}+\bm{\zeta})+ h(g)\nabla g&=\bm{f}, \label{1:E0:eul11}\\
(\p_t+(\bm{u}+\bm{\xi})\cdot\nabla)g+\nabla\cdot (\bm{u}+\bm{\xi})&=0,\label{1:E0:eul21}\\
\bm{u}(0)=\bm{u}_0, \quad g(0)&=g_0.\label{1:E0:eulic1}
\end{align}
We set $\bm U=(\bm{u}_{0}, g_{0}, \bm{\zeta}, \bm{\xi},\bm f)$,
\begin{align}\bm Y^k&= C(J_T,\bm{H}^k)\times C(J_T,H^k), \nonumber\\ \bm X^k&= \bm H^k\times H^k
\times L^2(J_T,\bm{H}^{k+1})\times L^2(J_T,\bm{H}^{k+1})\times
L^2(J_T,\bm{H}^{k}),\nonumber\end{align} and endow these spaces
with natural norms. Standard arguments show that if $k\ge4$, then
for any $\bm U\in \bm X^k$ problem
(\ref{1:E0:eul11})-(\ref{1:E0:eulic1})  has at most one solution
$(\bm u, g)\in\bm Y^k $. The following theorem establishes a
perturbative result on the existence of solution and some
continuity properties of the resolving operator.

\begin{theorem}\label{1:T:pert}
 Let $T>0$, $k\ge 4$ and $h\in C^k(\R)$ be such that $0<h(s)$ for any $s\in \R$.
 Suppose that for some  function   $\bm U_1 \in\bm X^k $
problem (\ref{1:E0:eul11})-(\ref{1:E0:eulic1})  has a solution
$(\bm{u}_1,g_1)\in  \bm Y^k$. Then  there are constants $\delta>0$
and $C>0$ depending only on $h$ and $ \|\bm U_1\|_{\bm X^k}$ such
that the following assertions hold.
\begin{enumerate}
\item[(i)] If $\bm{U}_{2}\in \bm{X}^{k}$
 satisfies the inequality
\begin{align}\label{1:E1:delt2}
\|\bm{U}_{1}-\bm{U}_{2}\|_{\bm{X}^{k}}< \delta,
\end{align}
then problem (\ref{1:E0:eul11})-(\ref{1:E0:eulic1}) has a unique
solution $ (\bm{u}_2,g_2)\in  \bm Y^k.$
\item[(ii)] Let
\begin{align}\RR:&\bm{X}^{k} \rightarrow  \bm Y^k\nonumber
\end{align} be the operator that takes a function
$\bm{U}_{2}$ satisfying (\ref{1:E1:delt2}) to the solution
$(\bm{u}_2,g_2)\in \bm Y^k$ of problem
(\ref{1:E0:eul11})-(\ref{1:E0:eulic1}). Then
\begin{align}\nonumber
&\|\RR(\bm{U}_{1})-\RR(\bm{U}_{2} )\|_{\bm Y^{k-1}}\le
C\|\bm{U}_{1}-\bm{U}_{2}\|_{\bm X ^{k-1}}.
\end{align}
\item[(iii)] The operator $\RR:\bm{X}^{k} \rightarrow  \bm Y^k$ is continuous at $\bm U_1$.
\end{enumerate}
\end{theorem}
We emphasize the fact that the constants $\de$ and $C$ depend only
on the norm of $\bm U_1$. This observation will be important in
Section \ref{1:S.4},  where we construct a solution of
(\ref{1:E0:eul11})-(\ref{1:E0:eulic1}) with the help of a
perturbative argument.
\begin{proof}
We seek  a solution of (\ref{1:E0:eul11})-(\ref{1:E0:eulic1}) in
the form $(\bm u_2,g_2):=(\bm u _1,g_1)+(\bm {w},\ph) $.
Substituting this into  (\ref{1:E0:eul11})-(\ref{1:E0:eulic1}) and
performing some transformations, we obtain the following problem:
\begin{align}
\p_t{\bm{w}}&+((\bm{u_1}+\bm{\zeta_1})\cdot\nabla)(\bm{w}+\bm{\eta})+
((\bm{w}+\bm{\eta})\cdot\nabla)(\bm{u_1}+\bm{\zeta_1})\nonumber\\&+((\bm{w}+\bm{\eta})\cdot\nabla)(\bm{w}+\bm{\eta})
 +h(g_1+\ph)\nabla (g_1+\ph)-h(g_1)\nabla g_1=\bm q,
 \label{1:EEE2c:eul13}\\
\p_t\ph&+((\bm{u_1}+\bm{\xi_1})\cdot\nabla)\ph+((\bm{w}+\bm{\sigma})\cdot\nabla)g_1+((\bm{w}+\bm{\sigma})\cdot\nabla)\ph
\nonumber\\&+\nabla\cdot (\bm{w}+\bm{\sigma})=0,\label{1:EEE2c:eul23}\\
&(\bm w,\ph)(0)=(\bm w_0,\ph_0),\label{1:EEE2:eulc3}
\end{align}
where $\bm{\eta}=\bm{\zeta_2}-\bm{\zeta_1}$,
$\bm{\sigma}=\bm{\xi_2}-\bm{\xi_1}$, $\bm q=\bm f_2-\bm f_1$, $\bm
w_0=\bm u_{20}-\bm u_{10}$ and $ \ph_0=g_{20}-g_{10}$.
       Problem (\ref{1:EEE2c:eul13})-(\ref{1:EEE2:eulc3}) is a quasi-linear symmetrizable hyperbolic system.
   Indeed, setting $V=\left(\!\!
           \begin{array}{cc}
             \bm w \\ \ph
           \end{array}\!\!
         \right)$ and $a_i^k=h(g_1+\ph)\delta_{i,k}$,
system    (\ref{1:EEE2c:eul13})-(\ref{1:EEE2:eulc3}) can be
rewritten in the form
  \begin{align}\label{1:E:symm}
\p_t{\bm{V}}&+\sum_{i=1}^3\bm A_i(t,\bm x, \bm V)\p_i \bm V+\bm
G(t,\bm x, \bm V)=0, &\bm V(0)=(\bm w_0,\ph_0),
\end{align}
where
\begin{align}
\bm A_i =\left(\!
                            \begin{array}{ccccc}\\ \!\!
                             (\bm{u_1}+\bm{\zeta_1}+\bm{w}+\bm{\eta})_i\!\!\! \!\!\!\!\!\!& 0 & 0 & a_1^i  \\
                              0 &\!\!\! \!\!\!\!\!\!(\bm{u_1}+\bm{\zeta_1}+\bm{w}+\bm{\eta})_i \!\!\! \!\!\!\!\!\!& 0 & a_2^i\\
                              0 & 0 &\!\!\! \!\!\!\!\!\!(\bm{u_1}+\bm{\zeta_1}+\bm{w}+\bm{\eta})_i \!\!\! \!\!\!\!\!\!& a_3^i \\
                              \delta_1^i & \delta_2^i & \delta_3^i & \!\!\! \!\!\!\!\!\!(\bm{u_1}+\bm{\zeta_1}+\bm{w}+\bm{\sigma})_i\!\! \end{array}
                          \right),\qquad\qquad\nonumber\\
\qquad\qquad\nonumber\\
                           \bm G(t,\bm x, \bm V)\qquad\qquad\qquad\qquad\qquad\qquad
                           \qquad\qquad\qquad\qquad\qquad\qquad\qquad\qquad\qquad
                           \qquad\qquad\quad
                           \nonumber\\ =\left(\!
                            \begin{array}{cc}
                                                   ((\bm{u_1}+\bm{\zeta_1})\cdot\nabla) \bm{\eta}+
((\bm{w}+\bm{\eta})\cdot\nabla)(\bm{u_1}+\bm{\zeta_1})+ (h(g_1+\ph)-h(g_1))\nabla g_1-\bm q\\
                             ((\bm{w}+\bm{\sigma})\cdot\nabla)g_1+\nabla\bm{\sigma}  \nonumber
                            \end{array}\!
                          \right).\qquad\qquad\nonumber
\end{align}
 Now note that (\ref{1:E:symm}) is symmetrizable hyperbolic system, since  \begin{align}
\bm A_0(t,\bm x, \bm V)  = \left(
                             \begin{array}{cccc}
                               1 & 0 & 0 & 0 \\
                               0 & 1 & 0 & 0 \\
                               0 & 0 & 1 & 0 \\
                               0 & 0 & 0 & h(g_1+\ph) \\
                             \end{array}
                           \right)
\end{align}
 is positive definite and $\bm A_0\cdot \bm A_i,$ $i=1,2,3$ are symmetric. By Theorem \ref{T:symeths}, there is a solution $ \bm V\in C(J_{T_0},\bm{H}^k)\times
 C(J_{T_0},H^k)$ of
 (\ref{1:E:symm}) for some   $T_0\le T$.
Now we prove that $T_0=T$.   First, let us rewrite system
(\ref{1:EEE2c:eul13}), (\ref{1:EEE2c:eul23}) in the form
\begin{align}
\p_t{\bm{w}}&+((\bm{u_1}+\bm{\zeta_1})\cdot\nabla)(\bm{w}+\bm{\eta})+
((\bm{w}+\bm{\eta})\cdot\nabla)(\bm{u_2}+\bm{\zeta_2})\nonumber\\&
 +h(g_1)\nabla \ph+(h(g_2)-h(g_1))\nabla g_2=\bm q,
 \label{1:EEE2c:eul143ew}\\
\p_t\ph&+((\bm{u_1}+\bm{\xi_1})\cdot\nabla)\ph+
((\bm{w}+\bm{\sigma})\cdot\nabla)g_2+\nabla\cdot
(\bm{w}+\bm{\sigma})=0.\label{1:EEE2c:eul243ew}
\end{align}
Taking the $\p^\al:=\frac{\p^\al}{\p x^\al},$ $|\al|\le k-1$
derivative of (\ref{1:EEE2c:eul143ew}) and multiplying the
resulting equation by $\p^\al \bm {w}$, we get
\begin{align}\label{1:EEE.rott1111}
\frac{1}{2}&\frac{\dd}{\dd t}\|\p^\al\bm{w}\|_0^2\nonumber\\&+
\int\p^\al\big(( (\bm{u_1}+\bm{\zeta_1})\cdot\nabla)
\bm{w}\big)\cdot\p^\al\bm {w}\dd \bm x+
\int\p^\al\big((\bm{w}\cdot\nabla)
(\bm{u_2}+\bm{\zeta_2})\big)\cdot\p^\al\bm {w}\dd \bm
x\nonumber\\&+ \int\p^\al\big( h(g_1)\nabla \ph\big)\cdot\p^\al\bm
{w}\dd \bm x +\int\p^\al\big((h(g_2)-h(g_1))\nabla
g_2\big)\cdot\p^\al\bm {w}\dd \bm x\nonumber\\&\le
C\|\bm{w}\|_{k-1}( \|\bm{\eta}\|_{k}+\|\bm q\|_{k-1}).
\end{align}
Integrating by parts, we see that
\begin{align}\label{1:EEE.intbyparts}
\int \p^\al\big(( (\bm{u_1}+\bm{\zeta_1})\cdot\nabla)
\bm{w}\big)\cdot\p^\al\bm {w}\dd \bm x\le&\int(
(\bm{u_1}+\bm{\zeta_1})\cdot\nabla) \p^\al\bm{w}\cdot\p^\al\bm
{w}\dd \bm x+C\|\bm{w}\|_{k-1}^2\nonumber\\=&-\frac{1}{2}\int(
\nabla\cdot(\bm{u_1}+\bm{\zeta_1}) ) |\p^\al\bm{w}|^2\dd \bm
x+C\|\bm{w}\|_{k-1}^2.
\end{align}
Inequalities  (\ref{1:EEE.rott1111}), (\ref{1:EEE.intbyparts}) and
the fact that  $ \bm H^k \hookrightarrow  \bm L^\infty$ for $k>
\frac{3}{2}$ imply that
\begin{align}\label{1:EEE.rott14111}
\frac{1}{2}\frac{\dd}{\dd t}\|\p^\al\bm{w}\|_0^2+ \int
h(g_1)\nabla \p^\al\ph\cdot\p^\al\bm {w}\dd \bm x\le&
C\|\bm{w}\|_{k-1}( \|\bm{w}\|_{k-1}+\|\ph\|_{k-1}\nonumber\\&+
\|\bm{\eta}\|_{k}+\|\bm q\|_{k-1}).
\end{align}
On the other hand, applying $\p^\al$ to (\ref{1:EEE2c:eul243ew}),
multiplying  the resulting equation by $h(g_1)\p^\al\ph$ and
integrating over $\T^3$, we obtain
\begin{align}
\frac{1}{2}&\frac{\dd}{\dd t}\int h(g_1) (\p^\al\ph)^2\dd \bm x-
\frac{1}{2}\int \p_t h(g_1) (\p^\al\ph)^2\dd \bm x\nonumber\\&+
\int \p^\al((\bm{u_1}+\bm{\xi_1})\cdot\nabla \ph)h(g_1)\p^\al\ph
\dd \bm x + \int \p^\al(\bm{w}\cdot\nabla g_2)h(g_1)\p^\al\ph \dd
\bm x\nonumber\\&+\int \p^\al(\nabla\cdot\bm{w}) h(g_1)\p^\al\ph
\dd \bm x\le C\|\ph\|_{k-1} \|\bm{\sigma}\|_{k}.\nonumber
\end{align}
As  $g_1\in C(J_T,H^k)$ and  $h\in C^k(\R)$,  integration by parts
in the  third  term on the left-hand side implies (cf.
(\ref{1:EEE.intbyparts}))
\begin{align}\label{1:EEE.rotft11111}
\frac{1}{2}\frac{\dd}{\dd t}\int h(g_1) (\p^\al\ph)^2\dd \bm
x+\int \p^\al(\nabla\cdot\bm{w}) h(g_1)\p^\al\ph \dd \bm x \le&
C\|\ph\|_{k-1}(\|\ph\|_{k-1}\nonumber\\&+\|\bm w\|_{k-1}+
\|\bm{\sigma}\|_{k}).
\end{align}
Adding (\ref{1:EEE.rott14111}) and (\ref{1:EEE.rotft11111}) and
using the facts that $h(s)>0$ for any $s\in \R$
\begin{align}\int \p^\al(\nabla\cdot\bm{w})
h(g_1)\p^\al\ph \dd \bm x +\int h(g_1)\nabla
\p^\al\ph\cdot\p^\al\bm {w} \dd \bm x=&-\int \p^\al \ph(\nabla
h(g_1)) \cdot\p^\al\bm {w} \dd \bm x,\nonumber
\end{align}
 we get
\begin{align}
\frac{\dd}{\dd t}\|\p^\al\bm{w}\|_0^2+  \frac{\dd}{\dd t}\int
(\p^\al\ph)^2\dd \bm x\le&
C(\|\bm{w}\|^2_{k-1}+\|\ph\|^2_{k-1}+\|\bm{\sigma}\|^2_{k}+
\|\bm{\eta}\|^2_{k}\nonumber\\&+\|\bm{q}\|_{k-1}^2).\nonumber
\end{align} Taking the sum over all $\al$, $|\al|\le k-1$ and applying the Gronwall inequality,    we obtain
\begin{align} \label{1:E:nortpah}
\|\bm{w}\|_{k-1}^2+  \|\ph\|_{k-1}^2 \le&
C(\|\bm{\sigma}\|^2_{L^2(J_T,H^k)}+ \|\bm{\eta}\|^2_{L^2(J_T,H^k)}
+\|\bm{q}\|_{L^2(J_T,H^{k-1})}^2).
\end{align}
   Thus we have that $T_0=T.$ Moreover, (\ref{1:E:nortpah})
completes also the proof of $(ii)$.

Assertion  $(iii)$ can be proved by repeating the arguments of the
proof of Theorem 1.4 in \cite{BDV}  for Sobolev spaces $\bm H^k$.
\end{proof}

  \subsection{ Continuity property of the resolving operator}

In this subsection, we establish another property of resolving
operator, which will play an essential role in Section
\ref{1:S.4.2}.

\begin{theorem}\label{1:T.te2} Let  $\bm \zeta _n$ and $ \bm \xi _n$ be   bounded sequences in $C(J_T,\bm H^{k+2})$ and  ${\bm\xi}_n$ be   such
that
\begin{align}
 \int_0^{t_0}\bm\xi_n(t)\cdot\bm \chi_n(t)\dd t\rightarrow 0 \text{ in }  H^{k}  \label{1:EE.xihatk}
\end{align}
for any $t_0\in J_T$ and for any uniformly equicontinuous sequence
$\bm \chi_n:J_T \rightarrow~\bm H^{k}$. Suppose that for $\bm
U_n=(\bm{u}_{0}, g_{0}, \bm{\zeta}_n, {\bm\xi}_n,\bm f) \in\bm
X^{k+1} $ problem (\ref{1:E0:eul11})-(\ref{1:E0:eulic1})  has a
solution $(\bm{u}_n,g_n)\in  \bm Y^{k+1}$.   Then for sufficiently
large $n\ge 1$ there exists a solution $\RR(\bm V_n)\in \bm
Y^{k+1}$ with $\bm V_n=(\bm{u}_{0}, g_{0}, \bm{\zeta}_n, 0,\bm f)
$, which  verifies
$$\RR(\bm U_n)-\RR(\bm V_n)\rightarrow 0
\text{ in } \bm Y^{k}. $$
\end{theorem}
\begin{proof}

%The proof is close to that of Theorem \ref{1:T:pert}.
   As $\bm V_n\in \bm X^{k+1}$, a blow-up criterion  for
quasi-linear symmetrizable hyperbolic systems \cite[Section 16,
Proposition 2.4]{tey1} implies that, if we have
 the
existence of $ \RR(\bm V_n)$ in $\bm Y^{k}$, then  $\RR(\bm
V_n)\in Y^{k+1}$.
 We seek the solution $ \RR(\bm V_n)$ in the form $(\bm w_n+\bm
u_n,\ph_n+g_n)$.  For $(\bm w_n,\ph_n)$ we have the following
problem (cf. (\ref{1:EEE2c:eul13})-(\ref{1:EEE2:eulc3}))
\begin{align}
\p_t{\bm{w}_n}&+((\bm{u}_n+\bm{\zeta_n})\cdot\nabla)\bm{w}_n+
(\bm{w}_n\cdot\nabla)(\bm{u}_n+\bm
{\zeta_n})\nonumber\\&+(\bm{w}_n\cdot\nabla)\bm{w}_n
 +h(g_n+\ph_n)\nabla (g_n+\ph_n)-h(g_n)\nabla g_n=0,
 \label{1:ETe2:eul13}\\
\p_t\ph_n&+(\bm{u}_n\cdot\nabla)\ph_n+(\bm{w}_n\cdot\nabla)g_n+(\bm{w}_n\cdot\nabla)\ph_n-\bm{\xi}_n\cdot\nabla
g_n
\nonumber\\&+\nabla\cdot (\bm{w}_n-\bm{\xi}_n)=0,\label{1:EEE2c:eul2843}\\
&(\bm w_n,\ph_n)(0)=(0,0).\label{1:ETE2:eulc3}
\end{align}
As $\|\bm{\xi}_n\cdot\nabla g_n\|_{k}+\|\nabla\cdot\bm{\xi}_n\|_k$
is not necessarily small, we cannot immediately conclude the
existence of  a solution $(\bm w_n,\ph_n)\in Y^{k}$.
 However, from the theory of the local existence of   solutions for quasi-linear symmetrizable
 hyperbolic systems  we have that for any   constant $\nu>0$ there is  a time   $T_{0,n}>0$   such that if
 $\|\tilde {\bm{w}}_n(0)\|_{k}+\|\tilde {\ph}_n(0)\|_{k}< \nu$, then problem (\ref{1:ETe2:eul13})-(\ref{1:EEE2c:eul2843})
  with initial data $(\tilde {\bm{w}}_n(0), \tilde \ph_n(0))$ has a  solution $(\bm w_n,\ph_n)\in Y^{k}$ on the interval    $[0,
  T_{0,n}]$.    Here time $T_{0,n}>0$ depends only on $\| \RR(\bm U_n)\|_{\bm Y^{k} }$ and
  $\nu$.
     Using estimation (\ref{1:E:nortpah}) and the fact that $\bm \zeta _n$ and $ \bm \xi _n$ are
   bounded sequences in $C(J_T,\bm H^{k+1})$, we get
   \begin{align*}
\|\bm{u_n}\|_{k}^2+  \|g_n\|_{k}^2 \le& C(\|\bm \zeta
_n\|^2_{L^2(J_T,H^{k+1})}+ \|\bm \xi _n \|^2_{L^2(J_T,H^{k+1})}
+\|\bm f\|_{L^2(J_T,H^k)}^2)\le C_1.
\end{align*}
Thus $\| \RR(\bm U_n)\|_{Y^k} $ is bounded and solutions $(\bm
w_n,\ph_n)$ are defined on the same interval $J_{T_0}$.   A simple
iterative argument shows that, to complete the proof, it
   suffices to prove that $\|\bm w_n\|_{C(T_0,H^k)}+\|\ph_n\|_{C(T_0,H^k)}<\nu$ for sufficiently large $n$. To this end,
   let us argue as in the proof of Theorem~ \ref{1:T:pert}. Taking the  $\p^\al$, $|\al|\le k$
   derivative of (\ref{1:ETe2:eul13}) and multiplying the resulting equation by  $\p^\al\bm w_n$ in $L^2$, we get
      (cf. (\ref{1:EEE.rott14111}))
\begin{align}\label{1:EEE.rott1411122}
\frac{1}{2}\frac{\dd}{\dd t}\|\p^\al\bm{w_n}\|_0^2+ \int
h(g_n)\nabla \p^\al\ph_n\cdot\p^\al\bm {w}_n\dd \bm x\le&
C\|\bm{w}_n\|_{k}( \|\bm{w_n}\|_{k}+\|\ph_n\|_{k}).
\end{align}
Then, applying $\p^\al$, $|\al|\le k$  to
(\ref{1:EEE2c:eul2843})   and multiplying the obtained equation by
$h(g_n)\p^\al\ph_n$, we derive
\begin{align}\label{1:EEE.rotft111112a}
\frac{1}{2}\frac{\dd}{\dd t}&\int h(g_n) (\p^\al\ph_n)^2\dd \bm
x+\int \p^\al(\nabla\cdot\bm{w}_n) h(g_n)\p^\al\ph_n \dd \bm x
\nonumber
\\ \le&\int h(g_n)\p^\al\ph_n\p^\al(\bm{\xi}_n\cdot\nabla g_n
+\nabla\cdot \bm{\xi}_n) \dd \bm x+
C\|\ph_n\|_{k}(\|\ph_n\|_{k}+\|\bm w_n\|_{k}).
\end{align}
Combining  (\ref{1:EEE.rott1411122}), (\ref{1:EEE.rotft111112a})
and the fact that
$$\int_0^{T_0} h(g_n)\p^\al\ph_n\p^\al(\bm{\xi}_n\cdot\nabla g_n +\nabla\cdot \bm{\xi}_n) \dd s\rightarrow 0   \text{ in }L^2(\T^3),$$
we get that   $\|\bm
w_n\|_{C(T_0,H^k)}+\|\ph_n\|_{C(T_0,H^k)}<\nu$ for sufficiently
large $n$.   Thus $\RR(\bm V_n)\in Y^{k}$ and
$$\|\RR(\bm U_n)-\RR(\bm V_n) \|_{\bm Y^{k}} \rightarrow 0. $$
\end{proof}

\section{Main results}
\subsection{Controllability of  Euler system}\label{1:subsect3.1}
Let us consider the controlled system associated with the
compressible Euler problem:
\begin{align}
\p_t{\bm{u}}+(\bm{u}\cdot\nabla)\bm{u}+ h(g)\nabla g&=\bm{f}+\bm{\eta}, \label{1:E0c:eul1}\\
(\p_t+\bm{u}\cdot\nabla)g+\nabla\cdot \bm{u}&=0,\label{1:E0c:eul2}\\
\bm{u}(0)=\bm{u}_0, \quad g(0)&=g_0,\label{1:E0:eulc}
\end{align}
where $\bm f\in C^\infty([0,\infty),\bm H^{k+2})$, $\bm u_0\in \bm
H^k$ and $g_0\in H^k$ are given functions, and~ $\bm \eta$ is the
control taking values in a finite-dimensional subspace $\bm
E\subset \bm H^{k+2}$. We denote by $\Theta(\bm u_0,g_0,\bm f)$
the set of functions  $\bm \eta \in L^2(J_T, \bm H^k)$ for which
problem (\ref{1:E0c:eul1})-(\ref{1:E0:eulc}) has a solution in $
\bm Y^k$.  For any $\alpha>0$ and $k\in \N$ let us define the set
$$G^k_\alpha=\{g\in H^k: \int e^{g(\bm x)}\dd \bm x=\alpha\}.$$
Recall that $\RR$ is  the resolving operator of
(\ref{1:E0:eul11})-(\ref{1:E0:eulic1}).    We denote by
$R_t(\cdot)$ the restriction of $R(\cdot)$ to the time $t$.  Let
$\bm X\subset L^2(J_T,\bm H^k)$ be an arbitrary vector
 subspace. We endow $G^k_\al$   with the
metric defined  by the norm of $H^k$ and $\bm X$ by the norm of
$L^2(J_T,\bm H^k)$. Recall that  for a function $f :\R^m
\rightarrow \R^n$ a point $x\in \R^m$ is said to be  regular point
if the  differential $D f(x) $ is surjective. Then, by the inverse
function theorem, there exists a neighborhood of $f(x)$ such that
a right inverse of $f $ is well defined. Now we give a
generalization of the notion of a regular point for a continuous
function $\bm F:\bm H^k\times G^k_\alpha\rightarrow \R^N$.

\begin{definition}\label{1:D.2.1} A point $( \bm {u_1},  g_1)$ is
 said to be regular for $\bm F$ if there is a non-degenerate closed ball
$\bm B\subset \R^N$ centred at $\bm {y_1} = \bm F( \bm {u_1},
g_1)$ and a continuous function $\bm G:\bm B\rightarrow \bm
H^k\times G^k_\alpha$ such that $\bm G(\bm{ y_1}) =( \bm {u_1},
g_1)$ and $\bm F(\bm G(\bm y)) = \bm y$ for any $\bm y\in \bm B$.
\end{definition}
\begin{definition}\label{1:D.3.1} System (\ref{1:E0c:eul1}), (\ref{1:E0c:eul2}) with $ \bm {\eta} \in \bm X$
 is said to be  controllable
at~time $T>0$ if for any constants $\e,~ \alpha>0$, for any
continuous function $\bm F:\bm H^k\times G^k_\alpha\rightarrow
\R^N$, for any initial data $(\bm u_0, g_0)\in \bm H^k\times
G^k_\alpha$ and for any regular point $( \bm {u_1},  g_1)$ for
$\bm F$ there is a control $\bm \eta \in \Theta(\bm u_0,g_0,\bm
f)\cap \bm X $ such that
\begin{align}
&\|\RR_T(\bm u_0,g_0,0,0, \bm {\eta})-( \bm {u_1},  g_1)\|_{\bm
H^k\times H^k}<\e,\nonumber\\& \bm F(\RR_T(\bm u_0,g_0,0,0, \bm
{\eta}))=\bm F( \bm {u_1},  g_1).\nonumber
\end{align}
\end{definition}
 Let us note that this concept of controllability is stronger than
the approximate controllability   and is   weaker than the  exact
controllability.   In the following example   the constructed
function admits a right inverse.
\begin{example} For any function $\bm z \in \bm H^k\times
G^k_\al$ we set $$\bm F(\bm z):=\int |\bm z(x)|^2\dd x.$$ Then for
any nonzero elements $\bm z_1 \in \bm H^k$ and $z_2 \in G^k_\al$
the point $\bm z=(\bm z_1,z_2)$ is regular for $\bm F$.
\end{example}

For any finite-dimensional subspace $\bm E\subset \bm H^{k+2}$, we
denote by $\FF(\bm E)$ the largest vector space $\bm F\subset \bm
H^{k+2}$ such that for any $\bm \eta_1\in \bm F$ there are vectors
$\bm \eta, \bm \zeta^1,\ldots ,\bm \zeta^n\in \bm E$  satisfying
the relation
\begin{eqnarray}
\bm \eta_1=\bm\eta-\sum_{i=1}^n (\bm \zeta^i \cdot\nabla)\bm
\zeta^i.\nonumber
\end{eqnarray}
We define $\bm {E_k}$ by the rule
\begin{eqnarray}\
\bm E_0=\bm E,\quad \bm E_n=\FF(\bm E_{n-1})\quad \textrm{for}
\quad n \geq 1,\quad \bm E_\infty=\bigcup_{n=1}^\infty \bm
E_n.\nonumber
\end{eqnarray}
The following theorem is the main result of this section.
\begin{theorem}\label{1:T.2.1}
Suppose $\bm f\in C^\infty([0,\infty),\bm H^{k+2})$. If $\bm
E\subset \bm H^{k+2}$ is a finite-dimensional subspace such that
$\bm E_\infty$ is dense in $\bm H^{k+1}$, then system
(\ref{1:E0c:eul1}), (\ref{1:E0c:eul2})  with $\bm\eta\in
C^\infty(J_T,\bm E)$ is
   controllable at time  $T>0$.
\end{theorem}
This theorem will be established in Section \ref{1:S.3.2}. We now
construct an example of a subspace $E$ for which the hypothesis of
Theorem \ref{1:T.2.1} is satisfied.

 Let us introduce the functions
\begin{align}
c^i_{\bm{m}}(\bm {x}) = \bm e_i \cos\langle \bm m, \bm x \rangle,
\,s^i_{\bm m}(\bm x) = \bm e_i \sin\langle \bm m, \bm x
\rangle,\quad i=1,2,3,\nonumber
\end{align}
where  $\bm m \in \Z^3$ and $\{\bm e_i\}$ is the standard basis in
$\R^3$.
\begin{lemma}\label{1:L.Lexam}
 If $\bm E=\textup{span}\{c^i_{\bm{m}},s^i_{\bm{m}} ,0\le m_j\le1,  i,j=1,2,3\}$, then the vector space  $\bm E_\infty$
is dense in $\bm H^{k}$ for any $k\ge 0$.
\end{lemma}

It is straightforward to see that $\dim \bm E=45$.
\begin{proof}[Proof of Lemma \ref{1:L.Lexam}]
It suffices to show that
\begin{equation}\label{1:E:Hayk}
\textup{span}\{c^i_{\bm{m}},s^i_{\bm{m}} ,\quad |\bm m|\le
2^{j}\}\subset \bm E_{j+1}\quad\text{for all}\quad j\ge0,
\end{equation}
where $|\bm{m}|=|m_1|+|m_2|+|m_3|$. We prove (\ref{1:E:Hayk}) by
induction. The case $j=0$ is clear. We shall prove
(\ref{1:E:Hayk}) for $j\ge 1$ assuming that it is true for any
$j'<j$. If ${n_i}\neq0$, then it is easy to see
\begin{align}
s^i_{2\bm{n}}(\bm x)& = -\frac{2}{n_i}c^i_{\bm{n}}(\bm x)\cdot\nabla c^i_{\bm{n}}(\bm x),\nonumber\\
 -s^i_{2\bm{n}}(\bm x)&= -\frac{2}{n_i}s^i_{\bm{n}}(\bm x)\cdot\nabla s^i_{\bm{n}}(\bm x),\nonumber\\
c^i_{2\bm{n}}(\bm x)&= -\frac{1}{n_i}(s^i_{\bm{n}}(\bm
x)-c^i_{\bm{n}}(\bm x))\cdot\nabla
 (s^i_{\bm{n}}(\bm x)-c^i_{\bm{n}}(\bm x)),\nonumber\\
-c^i_{2\bm{n}}(\bm x)&= -\frac{1}{n_i}(s^i_{\bm{n}}(\bm
x)+c^i_{\bm{n}}(\bm x))\cdot\nabla  (s^i_{\bm{n}}(\bm
x)+c^i_{\bm{n}}(\bm x)). \nonumber
\end{align}
Thus $s^i_{2\bm{n}}(\bm x), c^i_{2\bm{n}}(\bm x)\in \bm E_{j+1}$
for any $|\bm n|\le 2^{j-1}$, $n_i\neq 0$. If $n_i=0$,  without
loss of generality, we can assume $n_1\neq0$, then
\begin{align}
s^1_{2\bm{n}}(\bm x)+s^i_{2\bm{n}}(\bm x)& =
-\frac{2}{n_1}(c^1_{\bm{n}}(\bm x)+c^i_{\bm{n}}(\bm x))
\cdot\nabla (c^1_{\bm{n}}(\bm x)+c^i_{\bm{n}}(\bm x)), \label{1:E.n1neq01}\\
-s^1_{2\bm{n}}(\bm x)-s^i_{2\bm{n}}(\bm x)& =
-\frac{2}{n_1}(s^1_{\bm{n}}(\bm x)+s^i_{\bm{n}}(\bm x))
\cdot\nabla (s^1_{\bm{n}}(\bm x)+s^i_{\bm{n}}(\bm
x)).\label{1:E.n1neq02}
 \end{align}
As $\pm s^1_{2\bm{n}}(\bm x)\in  \bm E_{j+1}$ and the right-hand
sides of (\ref{1:E.n1neq01}), (\ref{1:E.n1neq02}) are in are in
 $\bm E_{j+1}$, we get $s^i_{2\bm{n}}(\bm x)\in \bm E_{j+1}$ for
any $|\bm n|\le 2^{j-1}$, $i=1,2,3$. In the same way, we can show
that $c^i_{2\bm{n}}(\bm x)\in \bm E_{j+1}$. Now take $\bm l\in
\Z^3, |\bm l| \le 2^{j}$
 and let us choose $\bm n\in \Z^3, |\bm n| \le 2^{j-1}$ and $\bm m\in \Z^3, |\bm m| \le 2^{j-1}$ such that
\begin{align}
\bm l=\bm n+\bm m \text{ and }    c^i_{\bm n-\bm m},s^i_{\bm n-\bm
m}\in \bm E_1.  \nonumber
 \end{align}
For example, if $l=(l_1,l_2,l_3)$ and $l_1$ is even, we can take
$$\bm n=(\frac{l_1}{2},
\bigg[\frac{l_2}{2}\bigg],l_3-\bigg[\frac{l_3}{2}\bigg])\quad
\text{and}\quad
 \bm m=(\frac{l_1}{2},l_2-\bigg[\frac{l_2}{2}\bigg],\bigg[\frac{l_3}{2}\bigg]).$$
 A similar representation holds if
  $l_2$ or   $l_3$  is even. On the other hand, if all $l_i$ are odd, then necessarily $l_i\le 2^j-1$, and we can take
$$\bm n=(l_1-\bigg[\frac{l_1}{2}\bigg],
\bigg[\frac{l_2}{2}\bigg],l_3-\bigg[\frac{l_3}{2}\bigg])\quad
\text{and}\quad
 \bm m=(\bigg[\frac{l_1}{2}\bigg],l_2-\bigg[\frac{l_2}{2}\bigg],\bigg[\frac{l_3}{2}\bigg]).$$
  Using the identities
\begin{align}
(s^i_{\bm{n}}(\bm x)\pm s^i_{\bm{m}}(\bm x)
)\cdot\nabla(s^i_{\bm{n}}(\bm x)\pm s^i_{\bm{m}}(\bm x) )&=
\frac{n_i}{2}s^i_{\bm{2n}}(\bm x)+\frac{m_i}{2}s^i_{2\bm{m}}(\bm
x) \pm \frac{l_i} {2}s^i_{\bm{l}}(\bm x)\nonumber\\&\pm
\frac{n_i-m_i}
{2}s^i_{\bm{n-m}}(\bm x),\nonumber\\
(s^i_{\bm{n}}(\bm x)\pm c^i_{\bm{m}}(\bm x)
)\cdot\nabla(s^i_{\bm{n}}(\bm x)\pm c^i_{\bm{m}}(\bm x) )&=
\frac{n_i}{2}s^i_{\bm{2n}}(\bm x)-\frac{m_i}{2}s^i_{2\bm{m}}(\bm
x) \pm \frac{l_i} {2}c^i_{\bm{l}}(\bm x)\nonumber\\&\pm
\frac{n_i-m_i} {2}c^i_{\bm{n-m}}(\bm x),\nonumber
\end{align}
we obtain that if $l_i\neq0$, then $s^i_{\bm{l}}(\bm
x),c^i_{\bm{l}}(\bm x)\in \bm E_{j+1}$ for any $|\bm l|\le 2^{j}$,
$i=1,2,3$.
 Arguing as above, we can easily prove that also in the case $l_i=0$
we have $s^i_{\bm{l}}(\bm x),c^i_{\bm{l}}(\bm x)\in \bm E_{j+1}$.
\end{proof}

\subsection{Proof of Theorem \ref{1:T.2.1}}\label{1:S.3.2}

We shall need the concept of $(\e,\bm u_0, g_0,\bm
K)$-controllability of the system. Let us fix  constants
$\e,\al>0$, an initial point $(\bm u_0, g_0)\in \bm H^k\times
G^k_\alpha$,
 a compact set $\bm K\subset \bm H^k\times G^k_\alpha$ and a vector space  $\bm X\subset L^2(J_T,\bm H^k)$.
\begin{definition} We say that system (\ref{1:E0c:eul1}), (\ref{1:E0c:eul2})  with $\bm\eta\in \bm X$ is
  $(\e,\bm u_0, g_0,\bm K)$-controllable at
  time $T>0$ if there is a continuous mapping
$$\bm\Psi:\bm K\rightarrow  \Theta(\bm
u_0,g_0,\bm f)\cap \bm X $$  such that
\begin{eqnarray}
\sup_{(\hat{\bm u},\hat g)\in \bm K} \|\RR_T(\bm u_0,g_0, 0,0, \bm
\Psi(\hat{\bm u}, \hat{g}))-(\hat{\bm u},\hat g)\|_{\bm H^k\times
H^k}<\e,\nonumber
\end{eqnarray}
where  $ \Theta(\bm u_0,g_0,\bm f) \cap \bm X$ is endowed with the
norm of $L^2(J_T,\bm H^k)$.
\end{definition}
The proof of Theorem \ref{1:T.2.1} is deduced from the following
result.
\begin{theorem}\label{1:T.2.121}  If $\bm E\subset \bm H^{k+2}$, $k\ge 4$
is a finite-dimensional subspace such that $\bm E_\infty$ is dense
in $\bm H^{k+1}$, then   for any  $\e>0$,
 $(\bm u_0, g_0)\in \bm H^k\times G^k_\alpha$
 and  $\bm K\subset \bm H^k\times G^k_\alpha$  system (\ref{1:E0c:eul1}), (\ref{1:E0c:eul2})
with $\bm\eta\in C^\infty(J_T,\bm E)$ is
  $(\e,\bm u_0, g_0,\bm K)$-controllable at time  $T>0$.
\end{theorem}
Taking this assertion for granted, let  us complete the proof of
Theorem \ref{1:T.2.1}.  Suppose  $\e$ and $ \alpha$ are  positive
constants, $\bm F:\bm H^k\times G^k_\alpha\rightarrow \R^N$ is a
continuous function and $( \bm{ u_1},g_1)$ is a regular point for
$\bm F$. Thus, there is a closed ball $\bm B\subset \R^N$ centred
at $\bm{ y_1} = F( \bm{ u_1},g_1)$ of radius $r>0$ and a
continuous function $\bm G:\bm B\rightarrow \bm H^k\times
G^k_\alpha$ such that $\bm G( \bm{ u_1},g_1) =( \bm{ u_1},g_1)$
and $\bm F(\bm G(\bm y)) = \bm y$ for any $\bm y\in \bm B$.
Without loss of generality, we can assume that $\bm B$ is such
that
\begin{eqnarray}\label{1:E1.astxanish2}\sup_{\bm y\in \bm B}\|\bm G(\bm y)-( \bm{ u_1},g_1)
\|_{\bm H^k\times H^k}\le \frac{\e}{2}.\end{eqnarray}   Let us
choose a constant $0<\e_0<\e$ such that
\begin{align}\label{1:E:fihatk1}
\|\bm F(\hat {\bm y})-\bm F(\tilde {\bm y})\|_{\R^N}<r
 \text{ for any }  \hat {\bm y}, \tilde {\bm y}\in B, \|\hat {\bm y}- \tilde {\bm y}\|_{\bm H^k\times H^k}\le \frac{\e_0}{2}.
 \end{align}
  Since $\bm K:=\bm G(\bm B)$ is a compact subset of
$\bm H^k\times G^k_\alpha$, Theorem \ref{1:T.2.121} implies that
there  is a continuous mapping $\bm\Psi :\bm K \rightarrow
\Theta(\bm u_0,g_0,\bm f)\cap \bm X$ such that
\begin{eqnarray}\label{1:E1.astxanish} \sup_{(\hat {\bm u}, \hat
g)\in \bm K} \|\RR_T(\bm u_0,g_0,0,0,\bm\Psi(\bm{\hat u},\hat
g))-(\hat {\bm u}, \hat g)\|_{\bm H^k\times H^k}< \frac{\e_0}{2}.
\end{eqnarray} Therefore, the
continuous mapping $$\bm\Phi:\bm B\rightarrow \R^N, \quad \bm y
\rightarrow \bm F\big(\RR_T(\bm u_0,g_0,0,0, \bm\Psi \circ \bm G
(\bm y))\big)$$ satisfies the inequality
\begin{eqnarray}
\sup_{\bm y\in \bm B} \|\bm \Phi (\bm y)- \bm y\|_{\R^N}=\sup_{\bm
y\in \bm B} \|\bm F\big(\RR_T(\bm u_0,g_0, 0,0,\bm\Psi \circ \bm G
(\bm y))\big)- \bm F(\bm G(\bm y))\|_{\R^N} <{r}.\nonumber
\end{eqnarray}
  Applying the  Brouwer
theorem, we see that the mapping $ \bm y \rightarrow \bm {
y_1}+\bm y -\bm\Phi(\bm y )$ from $\bm B$ to $\bm B$ has a fixed
point $ \overline{ \bm y}\in \bm B$. Thus
\begin{eqnarray}
 \bm F(\RR_T(\bm u_0,g_0,0,0, \bm\Psi \circ \bm G (\overline{\bm y})))=\bm F(\bm{u}_1, g_1).\nonumber\end{eqnarray}
Using (\ref{1:E1.astxanish2}) and (\ref{1:E1.astxanish}), we
obtain
\begin{multline}
\|\RR_T(\bm u_0,g_0,0,0,\bm\Psi(\bm G (\overline{ \bm
y})))-(\bm{u}_1, g_1)\|_{\bm H^k\times H^k}\\\le\|\RR_T(\bm
u_0,g_0,0,0,\bm\Psi(\bm G (\overline{ \bm y})))-\bm G (\overline{
\bm y})\|_{\bm H^k\times H^k}\nonumber+\|\bm G (\overline{ \bm
y})- (\bm{u}_1, g_1)\|_{\bm H^k\times H^k} <\e .\nonumber
\end{multline}
This completes the proof.

\section{Proof of Theorem \ref{1:T.2.121}}\label{1:S.4}

\subsection{Reduction to controllability with $\bm E_1$-valued controls }
Theorem \ref{1:T.2.121} is derived from the proposition below,
which is established in Subsection \ref{1:S.4.2}.
\begin{proposition}\label{1:P.2} Suppose that $\bm
E\subset \bm H^{k+2}$ is a finite-dimensional subspace.
 Then   system (\ref{1:E0c:eul1}), (\ref{1:E0c:eul2})  is
$(\e,\bm u_0, g_0,\bm K)$-controllable  with $\bm \eta \in
C^\infty(J_T,\bm E_1)$ if and only if it is  $(\e,\bm u_0, g_0,\bm
K)$-controllable  with $\bm \eta \in C^\infty(J_T,\bm E)$.
\end{proposition}
\begin{proof}[Proof of Theorem \ref{1:T.2.121}]
In view of Proposition \ref{1:P.2}, it suffices to prove that
there is an integer $N\ge 1$, depending only on $\e $, $\bm u_0 $,
$g_0$ and $\bm K$, such that (\ref{1:E0c:eul1}),
(\ref{1:E0c:eul2})  with $\bm\eta\in C^\infty(J_T,\bm E_N)$ is
$(\e,\bm u_0,g_0,\bm K)$-controllable  at time $T$. For any
$\mu>0$    and $(\hat {\bm {u}},\hat g)\in \bm K$  let us define
\begin{align}
\bm {u}_{\mu}(t; \hat {\bm {u}}) &= T^{-1}(te^{-\mu\Delta}\hat
{\bm { u}} + (T - t)e^{-\mu\Delta} \bm u_0),\label{1:E:umuikar}\\
g_{\mu}(t;\hat g)&=\ln(T^{-1}(te^{\varphi_\mu(\hat g)} + (T-
t)e^{\varphi_\mu(g_0)})), \label{1:E:gmuikar}
\end{align}
where  $\varphi_\mu(g)\in G^{k+1}_\alpha$ is such that
$\varphi_\mu(g)\rightarrow g$ as $\mu\rightarrow 0$ for all $g\in
G^{k}_\alpha$. For example, we can take
$$\varphi_\mu(g)=\ln(\frac{\alpha}{\int \exp(e^{-\mu\Delta}g(\bm x))\dd \bm x})+e^{-\mu\Delta}g .$$

  \vspace{6pt}  \textbf{Step 1.}
In this step, we show that there are controls $\bm \eta_\mu \in
C^\infty(J_T, \bm H^k)$ and  $\bm \xi_\mu \in C^\infty(J_T, \bm
H^{k+1})$ satisfying
\begin{align}\label{1:E:dzitaih2} \RR( e^{-\mu\Delta}\bm u_0 ,
\varphi_\mu(g_0),\bm \xi_\mu,\bm \xi_\mu,\bm \eta_\mu)=(\bm
u_{\mu},g_{\mu}).\end{align} We first construct $\bm \xi_\mu
 \in C^\infty(J_T, \bm
H^{k+1}) $ such that
\begin{align}\label{1:E:dzitaih}
\p_tg_{\mu}+ ((\bm u_{\mu}+\bm \xi_\mu)\cdot \nabla)
g_{\mu}+\nabla\cdot (\bm u_{\mu}+\bm \xi_\mu)=0.
\end{align}
To this end, let us multiply (\ref{1:E:dzitaih}) by $e^{g_{\mu}} $
and perform  some simple transformations. We get
\begin{align}\label{1:E:dzitaih22}
\nabla\cdot (e^{g_{\mu}}\bm \xi_\mu)=- \p_te^{g_{\mu}}-\nabla\cdot
(e^{g_{\mu}} \bm u_{\mu}).
\end{align}   We seek a solution of this equation in the
form $e^{g_{\mu}}\bm \xi_\mu=\nabla \psi_\mu$. Substituting this
into    (\ref{1:E:dzitaih22}), we get
$$\Delta\psi_\mu=- \p_te^{g_{\mu}}-\nabla\cdot
(e^{g_{\mu}} \bm u_{\mu}).$$ This equation   has a solution $\bm
\psi_\mu
 \in C^\infty(J_T, \bm
H^{k+2})$ if and only if the integral of the right-hand side over
$\T^3$ is zero. The definitions of $\bm u_{\mu}, g_{\mu} $ imply
that
$$\int( \p_te^{g_{\mu}}+\nabla\cdot
(e^{g_{\mu}} \bm u_{\mu}))\dd \bm x = \p_t \int e^{g_{\mu}}\dd \bm
x =\p_t \al=0.$$ Thus  (\ref{1:E:dzitaih22}) has a solution $\bm
\xi_\mu
 \in C^\infty(J_T, \bm
H^{k+1})$.    Since $\|e^{-\mu \Delta} \bm
u_0\|_{k+1},\|\ph_{\mu}(g_0)\|_{k+1}$ are bounded with respect to
$\mu\in (0,1)$, the constructions of $\bm {u}_{\mu}$ and $g_{\mu}$
 imply that   $\|\p_te^{g_{\mu}}-\nabla\cdot
(e^{g_{\mu}} \bm u_{\mu})\|_k$ is also bounded. Thus
$\|\psi_\mu\|_{k+2}$ is bounded, which implies the boundedness of
$\|\bm \xi_\mu\|_{k+1} $.
 If we define
\begin{align}\label{1:e:ettaiban}
\bm{\eta}_\mu&=\p_t{\bm{u}_\mu}+((\bm{u}_\mu+\bm
\xi_\mu)\cdot\nabla)(\bm{u}_\mu+\bm \xi_\mu)+ h(g_\mu)\nabla
g_\mu-\bm{f},
\end{align}
then $\bm{\eta}_\mu\in C^\infty(J_T, \bm H^{k}) $ and
(\ref{1:E:dzitaih2}) holds.

  \vspace{6pt}  \textbf{Step 2.}  Let us take some  functions $\bm \xi_\mu^\delta\in C^\infty(J_T, \bm
H^{k+1}) $ such that $\bm \xi_\mu^\delta(0)=\bm
\xi_\mu^\delta(T)=0$ and
 \begin{align}
\|\bm \xi_\mu^\delta-\bm \xi_\mu\|_{L^2(J_T,H^{k+1})}\rightarrow0
\text { as } \delta \rightarrow0.
 \end{align}
  Using the constructions of $\bm \xi_\mu,\bm
\eta_\mu$ and the fact that
\begin{align}
(\bm {u}_{\mu}(T; \hat {\bm {u}}), g_{\mu}(T;\hat
g))=(e^{-\mu\Delta}\hat {\bm { u}},\ph_\mu(\hat g)),\nonumber
\end{align}
we have
\begin{align}\label{1:E:avel1:1}\RR_T( e^{-\mu\Delta}\bm u_0 , \varphi_\mu(g_0),\bm
\xi_\mu,\bm \xi_\mu,\bm \eta_\mu)=(e^{-\mu\Delta}\hat {\bm {
u}},\ph_\mu(\hat g)).\end{align} On the other hand, $\bm
\xi_\mu^\delta(0)=\bm \xi_\mu^\delta(T)=0$ implies
\begin{align}\label{1:E:avel1:2}
\RR_T( e^{-\mu\Delta}\bm u_0 , \varphi_\mu(g_0),\bm
\xi^\delta_\mu,\bm \xi^\delta_\mu,\bm \eta_\mu)=\RR_T(
e^{-\mu\Delta}\bm u_0,\varphi_\mu(g_0),0,0,\bm \eta_\mu-\p_t\bm
\xi^\delta_\mu)\end{align}
 Then, by Theorem \ref{1:T:pert}, (\ref{1:E:avel1:1}) and (\ref{1:E:avel1:2}), we obtain
 \begin{align} \label{1:E:deltasahm}
&\| \RR_T( e^{-\mu\Delta}\bm u_0 , \varphi_\mu(g_0),\bm
\xi_\mu,\bm \xi_\mu,\bm \eta_\mu)-\RR_T( e^{-\mu\Delta}\bm u_0 ,
\varphi_\mu(g_0),\bm \xi^\delta_\mu,\bm \xi^\delta_\mu,\bm
\eta_\mu)\|_{\bm H^k\times H^k}\nonumber\\&= \|
(e^{-\mu\Delta}\hat {\bm { u}},\ph_\mu(\hat g))- \RR_T(
e^{-\mu\Delta}\bm u_0,\varphi_\mu(g_0), 0,0,\bm \eta_\mu-\p_t\bm
\xi^\delta_\mu)\|_{\bm H^k\times H^k} \rightarrow0 \end{align}
 as $\delta \rightarrow0.$  Clearly
\begin{align}\label{1:E:nyuisahmn}
\|(e^{-\mu\Delta}\bm u_0,\varphi_\mu(g_0))-(\bm u_0,g_0)\|_{\bm
H^k\times H^k}+\|(e^{-\mu\Delta}\hat {\bm { u}},\ph_\mu(\hat
g))-(\hat {\bm { u}},\hat g) \|_{\bm H^k\times H^k}\rightarrow0
\end{align}
 as $\mu \rightarrow0$. The fact that $\bm
E_\infty$ is dense in $\bm H^{k+1}$ implies that
\begin{align}\label{1:E:Nisahman} \sup _{(\hat {\bm u },\hat {g})\in \bm K}\|P_{\bm E_N}
(\bm \eta_\mu-\p_t\bm \xi^\delta_\mu) - (\bm \eta_\mu-\p_t\bm
\xi^\delta_\mu)\|_{L^2(J_T,\bm H^{k+1})} \rightarrow 0 \,\,
\text{as}\,\,N \rightarrow \infty.
\end{align}
  Since $\|\bm \xi_\mu\|_{k+1} $ is bounded uniformly with respect
to $\mu\in (0,1)$, equation (\ref{1:e:ettaiban}) implies that
$\|\bm \eta_\mu\|_k$ is also bounded. Taking time derivative of
(\ref{1:E:dzitaih22}), we can show the boundedness of $\|\p_t\bm
\xi_\mu\|_k$. Thus
  $\|( e^{-\mu\Delta}\bm u_0 , \varphi_\mu(g_0),\bm 0,\bm
0,\bm \eta_\mu-\p_t\bm \xi^\delta_\mu)\|_{\bm X^k}$ is  bounded
uniformly with respect to $\mu\in (0,1)$.  Hence, by  Theorem
\ref{1:T:pert}  and relations
(\ref{1:E:deltasahm})-(\ref{1:E:Nisahman}), a  solution $\RR (\bm
u_0,g_0,   0,0, P_{\bm E_N}(\bm \eta_\mu(\hat {\bm u},\hat g)))\in
~\bm Y^k$ exists for sufficiently large $N\ge1$ and   sufficiently
small $\delta, \mu>0$. Moreover,
$$ \sup _{(\hat {\bm u },\hat {g})\in \bm K}\|\RR_T
(\bm u_0,g_0,  0,0, P_{\bm E_N}(\bm  \eta_\mu(\hat {\bm u},\hat
g)-\p_t\bm \xi^\delta_\mu(\hat {\bm u},\hat g)) )-(\hat {\bm
u},\hat g) \|_{\bm H^k\times H^k}<\e.$$    From
(\ref{1:E:dzitaih22}) and   the constructions of $\bm u_{\mu}(\hat
u), g_{\mu}(\hat g)$, we have   $$\bm \xi _\mu~ :~(\hat{\bm
u},\hat g)\mapsto ~  \bm \xi _\mu (\hat{\bm u},\hat g),
\quad\p_t\bm \xi _\mu :(\hat{\bm u},\hat g)\rightarrow \bm \p_t\xi
_\mu (\hat{\bm u},\hat g)$$ are continuous from $\bm K$ to $\bm
L^2(J_T,\bm H^{k+1})$.
 Then
(\ref{1:e:ettaiban}) implies that mapping
$$P_{\bm E_N}(\bm \eta_{\mu}-\p_t\bm \xi^\delta_\mu)
(\cdot,\cdot):(\hat{\bm u},\hat g)\rightarrow P_{\bm E_N}(\bm
\eta_{\mu}(\cdot,\hat{\bm u},\hat g)-\p_t\bm
\xi^\delta_\mu(\hat{\bm u},\hat g))$$   is continuous from $\bm K$
to $\bm L^2(J_T,\bm H^k)$. The proof is complete.
\end{proof}

\subsection {Proof of Proposition \ref{1:P.2}} \label{1:S.4.2}
The proof of  Proposition \ref{1:P.2} is inspired by ideas from
\cite{agr1, agr2, shi1,shi2}.
 Let us  admit for the  moment the
following lemma.
\begin{lemma}\label{1:LL.reguly}
For any $(\bm u_0,g_0)\in \bm H^{k+2}\times H^{k+2}$,   for any
$\e>0$  and for any continuous mapping $\bm \Psi_1:\bm K
\rightarrow {\Theta}(\bm u_0,g_0,\bm f)\cap C^\infty (J_T,\bm
E_1)$ there is a   constant $\nu>0$ and a continuous mapping
$\bm{\Psi}:\bm K\rightarrow {\Theta}(\bm u_0,g_0,\bm f)\cap
C^\infty (J_T,\bm E)$ such that
\begin{eqnarray}
\sup_{(\hat{\bm u}, \hat g)\in \bm K}\|\RR_T( \tilde
{\bm{u}}_0,\tilde g_0,0,0,\bm \Psi(\bm{\hat u},\hat g))-\RR_T(
\tilde {\bm u}_0,\tilde g_0,0,0,\bm \Psi_1(\bm{\hat u},\hat
g))\|_{\bm H^k\times H^k}<\e\label{1:E..Hkpl2}
\end{eqnarray}
for any $(\tilde {\bm u}_0,\tilde g_0)\in \bm H^{k+2}\times
H^{k+2}$ with  $ \|\bm u_0-\tilde {\bm u}_0\|_{k}+\|g_0-\tilde
{g}_0\|_{k}<\nu$.
 \end{lemma}

Let $(\bm u_0,g_0)\in \bm H^{k}\times H^{k}$  and $\bm \Psi_1:\bm
K \rightarrow {\Theta}(\bm u_0,g_0,\bm f)\cap C^\infty (J_T,\bm
E_1)$ be such that\begin{eqnarray}\label{1:E.prop2ipaym}
\sup_{(\hat{\bm u},\hat g)\in \bm K} \|\RR_T(\bm u_0,g_0, 0,0,\bm
\Psi_1(\hat{\bm u}, \hat{g}))-(\hat{\bm u},\hat g)\|_{\bm
H^k\times H^k}<\frac{\e}{2}.
\end{eqnarray} Take
 any sequence $(\bm
u^n_0,g^n_0)\in \bm {H}^{k+2}\times H^ {k+2}$ such that $$ \|(\bm
u_0,g_0)-(\bm u^n_0,g^n_0)\|_{\bm H^k\times H^k}\rightarrow 0
\text{ as } n\rightarrow\infty.$$ As $\bm K$ is compact, Theorem
\ref{1:T:pert} implies that $\bm{\Psi}_1(\bm K)\subset{\Theta}(\bm
u^n_0,g^n_0,\bm f)$ for sufficiently large $n$. By Lemma
\ref{1:LL.reguly}, there is a continuous mapping
$$\bm{\Psi}:\bm K\rightarrow {\Theta}(\bm u^n_0,g^n_0,\bm f)\cap
C^\infty (J_T,\bm E) $$ such that
\begin{eqnarray}
\sup_{(\hat{\bm u}, \hat g)\in \bm K}\|\RR_T(\bm
u^n_0,g^n_0,0,0,\bm {\Psi}(\bm{\hat u},\hat g)) -\RR_T(\bm
u^n_0,g^n_0,0,0,\bm {\Psi}_1(\bm{\hat u},\hat g))\|_{\bm H^k\times
H^k}<\frac{\e}{2}.\nonumber
\end{eqnarray}
 Choosing $n $ sufficiently large and using the fact that $\RR$ is uniformly continuous on the compact set $\bm {\Psi}(\bm K)\cup\bm {\Psi}_1(\bm K)$, we get
\begin{align}\label{1:E3.un1sd1}
\sup_{(\hat{\bm u}, \hat g)\in \bm K}\|&\RR_T(\bm u_0,g_0,0,0,\bm
{\Psi}(\bm{\hat u},\hat g))- \RR_T(\bm u_0,g_0,0,0,\bm
\Psi_1(\bm{\hat u},\hat g))\|_{\bm H^k\times
H^k}\nonumber\\
\le&\sup_{(\hat{\bm u}, \hat g)\in \bm K}\| \RR_T(\bm
u_0,g_0,0,0,\bm {\Psi}(\bm{\hat u},\hat g))- \RR_T(\bm
u^n_0,g^n_0,0,0,\bm {\Psi}(\bm{\hat u},\hat g))\|_{\bm H^k\times
H^k} \nonumber\\ & +\sup_{(\hat{\bm u}, \hat g)\in \bm K}
\|\RR_T(\bm u^n_0,g^n_0,0,0,\bm {\Psi}_1(\bm{\hat u},\hat
g))-\RR_T(\bm u_0,g_0,0,0,\bm \Psi_1(\bm{\hat u},\hat g))\|_{\bm
H^k\times H^k}\nonumber\\& +\sup_{(\hat{\bm u}, \hat g)\in \bm K}
\|\RR_T(\bm u^n_0,g^n_0,0,0,\bm {\Psi}(\bm{\hat u},\hat g))-
\RR_T(\bm u^n_0,g^n_0,0,0,\bm {\Psi}_1(\bm{\hat u},\hat g))\|_{\bm
H^k\times H^k}\nonumber\\<&\frac{\e}{2}.
\end{align}
Combining (\ref{1:E.prop2ipaym}) and (\ref {1:E3.un1sd1}), we
complete the proof of  Proposition \ref{1:P.2}.

\begin {proof}[Proof of Lemma \ref{1:LL.reguly}]

  \vspace{6pt}  \textbf{Step 1.}
 We shall need
the following lemma, which    can be proved by literal repetition
of the arguments  of the proof of \cite[Lemma 3.5]{shi2}.
\begin {lemma}\label{1:L3.1} For any
continuous mapping $\bm\Psi_1:\bm K\rightarrow  \Theta(\bm
u_0,g_0,\bm f)\cap L^2(J_T,\bm E_1) $  there is a  set $\bm A =
\{\eta_1^l , l = 1,  \ldots ,m\} \subset~\bm E_1$ an integer
$s\ge1$ and a mapping $\bm {\Psi}_s:\bm K\rightarrow \Theta(\bm
u_0,g_0,\bm f)\cap L^2(J_T,\bm E_1)$ such that
$$ \bm {\Psi}_s(\hat{\bm u },\hat g)=\sum_{l=1}^m  \sum_{r=0}^{
s-1}c_{l,r}(\hat{\bm u },\hat g)I_{r,s}(t) \bm\eta_1^l,$$ where
$c_{l,r}$ are non-negative functions such that
$\sum_{l=1}^mc_{l,r}=1$,  $I_{r,s}$ is the indicator function of
the interval $[t_r,t_{r+1})$ with $t_r=rT/s$
 and
\begin{eqnarray}
\sup_{(\hat{\bm u},\hat g)\in \bm K} \|\RR_T(\bm u_0,g_0,0,0, \bm
{\Psi}_1(\hat{\bm u},\hat g))-\RR_T(\bm u_0,g_0, 0,0,\bm
{\Psi}_s(\hat{\bm u},\hat g))\|_{\bm H^k\times H^k}<\e.\nonumber
\end{eqnarray}
\end {lemma}
Let $\bm {\Psi}_s$ be the function constructed in Lemma
\ref{1:L3.1}:
$$ \bm {\Psi}_s(\hat{\bm u },\hat g)=\sum_{l=1}^m\ph_l(t,\hat{\bm u },\hat g)\bm\eta_1^l.$$
  We claim that   there are vectors  $\bm\zeta^{l,1}, \ldots ,
\bm\zeta^{l,2n}, \bm\eta^{l} \in \bm E$ and positive constants
$\lambda_{l,1}, \ldots , \lambda_{l,2n}$ whose sum is equal to 1
such that
\begin{align}\label{1:E.lemma1}
& \bm\zeta^{i}=-\bm\zeta^{i+n} \text{ for }
 i=1,\ldots,n,\nonumber\\
&(\bm u\cdot \nabla)\bm u-\bm\eta_1^l= \sum_{j=1}^{2n}
\lambda_{l,j}
 ((\bm u+\bm\zeta^{l,j})\cdot \nabla)(\bm u+\bm\zeta^{l,j}) -\bm\eta^{l}\text{ for any }\bm u \in \bm H^1.
\end{align}
 Indeed, by the definition of $\FF(\bm E)$,
for any $\bm \eta_1^l\in \FF(\bm E)$ there are $\bm
\xi^{l,1},\ldots ,\bm \xi^{l,n},\bm \eta^l \in\bm  E$   such that
$$
\bm \eta^l_1=\bm \eta^l-\sum_{i=1}^n (\bm \xi^{l,i}\cdot\nabla\bm
\xi^{l,i}).$$ Let us set
$$
\lambda^{l,i} = \lambda^{l,i+n} = \frac {1}{2n},
\quad\bm\zeta^{l,i} = -\bm\zeta^{l,i+n} = \sqrt{n}\bm\xi^i, \quad
i = 1, \ldots ,n.
$$
Then (\ref{1:E.lemma1}) holds for any $\bm u \in \bm H^1$.

Let $(\bm u_1,g_1)=\RR(\bm u_0,g_0,0,0,\bm {\Psi}_s(\hat{\bm
u},\hat g))$. It follows from (\ref{1:E.lemma1})  that $(\bm
u_1,g_1)$ satisfies the problem
\begin{align}
\dot{\bm u}_1+\sum_{j=1}^{2n}\sum_{l=1}^m
\lambda_{l,j}\ph_l(t,\hat{\bm u},\hat
 g) &((\bm u_1+\bm\zeta^{l,j})\cdot
\nabla)(\bm u_1+\bm\zeta^{l,j})+h(g_1)\nabla g_1
 \nonumber\\& =\bm f(t)+\sum_{l=1}^m\ph_l(t,\hat{\bm u},\hat
 g)\bm\eta^l,\label{1:E3.aftleem1} \\
(\p_t+\bm{u_1}\cdot\nabla)g_1+\nabla\cdot \bm{u_1}&=0.\nonumber
\end{align}
  Taking $q=m\cdot n,$ $\{\bm\zeta^{i} \}_{i=1}^{q}
:=\{\bm\zeta^{l,j}\}_{l=1}^{m},_{j=1}^{n}$,
$\bm\zeta^{i+q}:=-\bm\zeta^{i} $, $ i=1,\ldots,q$, we rewrite
(\ref{1:E3.aftleem1})  in the form
\begin{align}\label{1:E3.aftleem3}
\dot{\bm u}_1&+\sum_{i=1}^{2q} \psi_i(t,\hat{\bm u},\hat g) ((\bm
u_1+\bm\zeta^{i})\cdot \nabla)(\bm u_1+\bm\zeta^{i})+h(g_1)\nabla
g_1 =\bm f(t)+\bm\eta (t,\hat{\bm u},\bm\hat g),
\end{align}
where
\begin{align}
&\bm\eta(t,\hat{\bm u},\hat g)=\sum_{l=1}^m\ph_l(t,\hat{\bm
u},\hat g)\bm\eta^l,\nonumber\\&
 \psi_i(t,\hat{\bm u}, \hat g
)=\sum_{r=0}^{s-1} d_{i,r}(\hat{\bm u},\hat g )I_{r,s}(t),
\end{align}
and $d_{i,r}\in C(\bm K)$ are  some non-negative functions such
that $$ \sum_ {i=1}^{q}d_{i,r}=\sum_
{i=q+1}^{2q}d_{i,r}=\frac{1}{2}.$$

  \vspace{6pt} \textbf{Step 2.}
   Let us show that it suffices to consider the case $s=1$.
Indeed, let us assume that for any constant $\e_0>0$ and for any
interval $I_r:=[t_{r-1},t_{r}]$ there exists  a continuous mapping
$\bm{\Psi}_{\e_0}^r:\bm K\rightarrow {\Theta}^r(\bm
u_1(t_{r-1}),g_1(t_{r-1}))\cap C^\infty (J_T,\bm E)$ such that
\begin{eqnarray}
\sup_{(\hat{\bm u}, \hat g)\in \bm K}\|\RR_{t_{r}-t_{r-1}} (\bm
u_1(t_{r-1}),g_1(t_{r-1}),0,0,\bm \Psi_{\e_0}^{r}(\bm{\hat u},\hat
g))-(\bm u_1(t_{r}),g_1(t_{r})) \|_{\bm H^k\times
H^k}<\e_0.\nonumber
\end{eqnarray}
%for any $(\tilde {\bm u}_0,\tilde g_0)\in \bm H^{k+2}\times
%H^{k+2}$  with $\bm \Psi^r(\bm K)\subset{\Theta}^r(\tilde {\bm u}_0,\tilde g_0)  $.
 Here ${\Theta}^r(\bm u_1(t_{r-1}),g_1(t_{r-1}))$ is the set of functions  $\bm \eta \in L^2(I_r, \bm H^k)$ for which
problem (\ref{1:E0c:eul1})-(\ref{1:E0c:eul2}) has a solution in $
C(I_r,\bm{H}^k)\times C(I_r,H^k)$ satisfying the initial condition
$$ \bm u(t_{r-1})=\bm u_1(t_{r-1}),\,\,g(t_{r-1})=g_1(t_{r-1}).$$
In view of Theorem \ref{1:T:pert}, there is   $\delta_s>0$ such
that for any $(\tilde {\bm u}_0,\tilde g_0)\in \bm H^{k+2}\times
H^{k+2} $ with $\|(\tilde {\bm u}_0,\tilde g_0)-(\bm
u_1(t_{s-1}),g_1(t_{s-1}))\|_{\bm H^{k}\times H^{k}}<\delta_s$ we
have the   inequality
 \begin{eqnarray}
\sup_{(\hat{\bm u}, \hat g)\in \bm K}\|\RR_{T-t_{s-1}} (\tilde
{\bm u}_0,\tilde g_0,0,0,\bm \Psi_\e^{s}(\bm{\hat u},\hat g))-(\bm
u_1(T),g_1(T)) \|_{\bm H^k\times H^k}<\e.\nonumber
\end{eqnarray}
Similarly, we can find $\delta_r>0$, $r=s-1,\dots,1$ such that
 \begin{eqnarray}
\sup_{(\hat{\bm u}, \hat g)\in \bm K}\|\RR_{t_{r+1}-t_{r}} (\tilde
{\bm u}_0,\tilde g_0,0,0,\bm \Psi_{\delta_{r+1}}^{r}(\bm{\hat
u},\hat g))-(\bm u_1(t_{r+1}),g_1(t_{r+1})) \|_{\bm H^k\times
H^k}<\delta_{r+1}\nonumber
\end{eqnarray}
  for any $(\tilde {\bm u}_0,\tilde g_0)\in \bm H^{k+2}\times H^{k+2} $ satisfying $$\|(\tilde {\bm u}_0,\tilde g_0)-(\bm u_1(t_r),g_1(t_r))\|_{\bm H^{k}\times H^{k}}<\delta_r.$$
Let us denote by $\hat{ \bm \Psi}: K\rightarrow L^2(J_T,\bm E)$
the continuous operator defined by the relations
\begin{align}
\hat {\bm \Psi}(\bm{\hat u},\hat g)(t)= \bm
\Psi_{\delta_{r+1}}^r(\bm{\hat u},\hat g)(t)   \text { for }  t\in
I_r, \nonumber
\end{align}
where $\delta_{s+1}=\e$. Then
\begin{eqnarray}
\sup_{(\hat{\bm u}, \hat g)\in \bm K}\|\RR_{T}( \tilde
{\bm{u}}_0,\tilde g_0,0,0,\hat{\bm \Psi}(\bm{\hat u},\hat g))-(\bm
u_1(T),g_1(T)) \|_{\bm H^k\times H^k}<\e.\nonumber
\end{eqnarray}
 To complete the proof,  it suffices to approximate  $\hat{\bm \Psi}$     in $L^2(J_T, \bm H^k) $  by a continuous mapping  $\bm{\Psi}:\bm
K\rightarrow {\Theta}(\bm u_0,g_0)\cap C^\infty (J_T,\bm E)$.

  \vspace{6pt} \textbf{Step 3.}    We now assume that  $s=1$. Then  (\ref{1:E3.aftleem3})
  takes the form
\begin{align}\label{1:E3.case1}
\dot{\bm u}_1&+\sum_{i=1}^{2q} d_i(\hat{\bm u},\hat g) ((\bm
u_1+\bm\zeta^{i})\cdot \nabla)(\bm u_1+\bm\zeta^{i})+h(g_1)\nabla
g_1 =\bm f(t)+\bm\eta (\hat{\bm u}, \hat g),
\end{align}
where $d_i\in C(\bm K)$ and $\bm\eta\in C(\bm K,\bm E) $.    For
any $n\in \N$, let $\bm\zeta_n(t,\hat{\bm u}, \hat
g)=\bm\zeta(\frac {nt}{T},\hat{\bm u}, \hat g)$, where
$\bm\zeta(t,\hat{\bm u}, \hat {g})$  is a $1$-periodic function
such that
$$\bm\zeta(s,\hat{\bm u}, \hat g)=\bm\zeta^j\text{ for }0\le
s-(d_1(\hat{\bm u}, \hat g)+\ldots+d_{j-1}(\hat{\bm u}, \hat
g))<d_j(\hat{\bm u}, \hat g),\quad j=1,\ldots ,q.$$ Note that
$\bm\zeta(t,\hat{\bm u}, \hat g)=-\bm\zeta(t-\frac{1}{2},\hat{\bm
u}, \hat g) $ for $t\in(\frac{1}{2},1)$. Eq.~(\ref{1:E3.case1}) is
equivalent to
\begin{align}
\dot {\bm u}_1  & + ((\bm u_1 + \bm\zeta_n(t,\hat{\bm u}, \hat
g))\cdot\nabla)(\bm u_1 + \bm\zeta_n(t,\hat{\bm u}, \hat
g))+h(g_1)\nabla g_1 \nonumber\\& = \bm f  + \bm\eta(t,\hat{\bm
u}, \hat g) + \bm {f_n}(t,\hat{\bm u}, \hat g),\nonumber
\end{align}
where
\begin{align}
\bm {f_n}(t,\hat{\bm u},\hat g)&= ((\bm u_1 +
\bm\zeta_n(t,\hat{\bm u}, \hat g))\cdot\nabla)(\bm u_1 +
\bm\zeta_n(t,\hat{\bm u}, \hat g)) \nonumber\\&-\sum_{i=1}^{2q}
d_i(\hat{\bm u}, \hat g) ((\bm u_1+\bm\zeta^{i})\cdot \nabla)(\bm
u_1+\bm\zeta^{i}).\nonumber
\end{align}
 Let us define
$$\KK  \bm {f_n}(t)=\int_0^t   \bm {f_n}(s)\dd s.$$
Then $\bm v_n = \bm u_1 - \KK \bm {f_n} $ is a solution of the
problem
\begin{align}
\dot {\bm v}_n + (({\bm v}_n + \bm{\zeta_n}(t,\hat{\bm u}, \hat
g)+ \KK \bm{f_n}(t,\hat{\bm u}, \hat g))\cdot \nabla)(\bm v_n +
\bm\zeta_n (t,\hat{\bm u}, &\hat g)+  \KK \bm{f_n}(t,\hat{\bm u},
\hat g))\nonumber\\+h(g_1)\nabla g_1 &= \bm f(t)
+ \bm\eta(t,\hat{\bm u}, \hat g),\nonumber\\
(\p_t+({\bm v}_n+\KK \bm{f_n}(t,\hat{\bm u}, \hat
g))\cdot\nabla)g_1+\nabla\cdot ({\bm v}_n+\KK \bm{f_n}(t,\hat{\bm
u}, \hat
g))&=0,\nonumber\\
\bm v_n& =\bm u_0.\nonumber
 \end{align}
 It is straightforward to see that
\begin{equation}
 \sup_{(\hat{\bm u}, \hat g)\in \bm K} \|\KK \bm f_n(t,\hat{\bm u}, \hat g)\|_{C(J_T,\bm H^{k+1})}\rightarrow 0,\nonumber
\end{equation}
(e.g. see  \cite[Chapter 3]{Jur} or \cite[Section 6]{hn}). Thus
\begin{equation} \label{1:E:p161}
  \|\bm v_n-\bm u_1\|_{C(J_T,\bm H^{k+1})}\rightarrow 0 \,\, \text{as} \,\,
  n\rightarrow \infty.
\end{equation}
On the other hand, Theorem \ref{1:T:pert} implies that
\begin{equation}\label{1:E:p162}
  \|(\bm v_n,g_1)-(\tilde{\bm u}_n,\tilde{g}_n)\|_{\bm Y^{k}}\rightarrow 0 \,\, \text{as} \,\,
  n\rightarrow \infty,
\end{equation}
where $(\tilde{\bm u}_n,\tilde{g}_n)$ satisfies the problem
\begin{align}
&\p_t {\tilde{\bm u}}_n+ ((\tilde{\bm u}_n + \bm\zeta_n
(t,\hat{\bm u}, \hat g))\cdot \nabla)(\tilde{\bm u}_n + \bm\zeta_n
(t,\hat{\bm u}, \hat g) )+h(\tilde{g}_n)\nabla \tilde{g}_n
 = \bm f(t)
+ \bm\eta(t,\hat{\bm u}, \hat g),\nonumber\\& (\p_t+\tilde{\bm
u}_n\cdot\nabla)\tilde{g}_n+\nabla\cdot \tilde{\bm
u}_n=0,\nonumber\\
&\tilde{\bm u}_n(0)=\tilde{\bm u}_0,\quad \tilde{g}_n(0)=\tilde
{g}_0.\nonumber
 \end{align}
We want to apply Theorem \ref{1:T.te2} to the above system. To
this end, let $\bm \chi_n:J_T \rightarrow\bm H^{k}$ be a uniformly
equicontinuous sequence  and let $t_0\in J_T$. Then
\begin{align} \int_0^{t_0}\bm\zeta_n(t)\cdot\bm
\chi_n(t)\dd t&=\int_0^{t_0}\bm\zeta(\frac{nt}{T})\cdot\bm
\chi_n(t)\dd t=\int_0^{\frac{nt_0}{T}}\bm\zeta(t)\cdot\bm
\chi_n(\frac{tT}{n})\frac{T}{n}\dd t\nonumber\\&=
\sum_{i=0}^{[\frac{nt_0}{T}]-1}\int_i^{i+1}\bm\zeta(t)\cdot\bm
\chi_n(\frac{tT}{n})\frac{T}{n}\dd
t+\int_{[\frac{nt_0}{T}]}^{\frac{nt_0}{T}}\bm\zeta(t)\cdot\bm
\chi_n(\frac{tT}{n})\frac{T}{n}\dd t.\label{1:E..hipotstug1}
\end{align} Using the construction of $\bm\zeta(t)$, we get
\begin{align}
\int_i^{i+1}\bm\zeta(t)\cdot\bm \chi_n(\frac{tT}{n})\dd
t&=\int_i^{i+\frac{1}{2}}\bm\zeta(t)\cdot\bm
\chi_n(\frac{tT}{n})\dd
t+\int_{i+\frac{1}{2}}^{i+1}-\bm\zeta(t-\frac{1}{2})\cdot\bm
\chi_n(\frac{tT}{n})\dd
t\nonumber\\&=\int_i^{i+\frac{1}{2}}\bm\zeta(t)\cdot\big(\bm
\chi_n(\frac{tT}{n})-\bm \chi_n(\frac{tT}{n}+\frac{T}{2n})\big)\dd
t.\nonumber
\end{align}
As $\bm \chi_n$ is  uniformly equicontinuous and $\bm\zeta$ is
bounded, we have
$$\sup_{t\in [0,n]}\|\bm\zeta(t)\cdot\big(\bm
\chi_n(\frac{tT}{n})-\bm
\chi_n(\frac{tT}{n}+\frac{T}{2n})\big)\|_k\rightarrow 0,\quad
n\rightarrow \infty.$$  The boundedness of $\bm\zeta\cdot\bm
\chi_n$ implies that the second term of the right-hand side of
(\ref{1:E..hipotstug1}) goes to zero. Thus
$$\int_0^{t_0}\bm\zeta_n(t)\cdot\bm \chi_n(t)\dd t\rightarrow 0   \text{ in $\bm H^k $} .$$
Using Theorem \ref{1:T.te2} and   limits (\ref{1:E:p161}),
(\ref{1:E:p162}), we get
\begin{equation}
 \sup_{(\hat{\bm u},\hat g)\in \bm K}\|\RR_T(\bm u_0,g_0,0,0,\bm\zeta_n ,\bm\zeta_n ,
 \bm\eta(\hat{\bm u}, \hat g))-(\bm u_1(T,\hat{\bm u}, \hat g),g_1(T,\hat{\bm u}, \hat g))\|_{\bm H^k\times H^k}<\e\nonumber
\end{equation}
for sufficiently large $n$. Let us take some functions $
{\bm\zeta}^m_n\in C^\infty(J_T,E)$ such that $ {\bm\zeta}_n^m(0)=
{\bm\zeta}_n^m(T)=0$ and
\begin{align}
\| {\bm\zeta}_n^m-{\bm\zeta}_n\|_{L^2(J_T,\bm H^{k+1})}\rightarrow
0 \text { as } m\rightarrow\infty .
\end{align}
Then Theorem \ref{1:T:pert} implies
\begin{equation}
 \sup_{(\hat{\bm u},\hat g)\in \bm K}\|\RR_T(\bm u_0,g_0,0,0,\bm\zeta_n ,\bm\zeta_n ,
 \bm\eta(\hat{\bm u}, \hat g))-\RR_T(\bm u_0,g_0,0,0,\bm\zeta^m_n ,\bm\zeta^m_n ,
 \bm\eta(\hat{\bm u}, \hat g))\|_{\bm H^k\times H^k}<\e.\nonumber
\end{equation}
For $m \gg1$, the operator
$${\Psi}:\bm K \rightarrow L^2(J_T,\bm E) ,\,\,\, (\hat{\bm u},\hat g)
\rightarrow \bm\eta(\hat{\bm u}, \hat g)-\p_t  {\bm\zeta}^m_n$$
satisfies
\begin{eqnarray}
\sup_{(\hat{\bm u}, \hat g)\in \bm K}\|\RR_T( {\bm{u}}_0,
g_0,0,0,\bm \Psi_1(\bm{\hat u},\hat g))-\RR_T( {\bm u}_0,
g_0,0,0,\bm {\Psi}(\bm{\hat u},\hat g))\|_{\bm H^k\times
H^k}<\e,\nonumber
\end{eqnarray}
which completes the proof.
\end {proof}

%\addcontentsline{toc}{section}{References}

\bibliographystyle{plain}

\begin{thebibliography}{10}





\bibitem{agr1}
A.~Agrachev and A.~Sarychev.
\newblock {Navier--Stokes equations controllability by means of low modes
  forcing}.
\newblock {\em J. Math. Fluid Mech.}, 7:108--152, 2005.

\bibitem{agr2}
A.~Agrachev and A.~Sarychev.
\newblock {Controllability of 2D Euler and Navier--Stokes equations by
  degenerate forcing}.
\newblock {\em Comm. Math. Phys.}, 265(3):673--�697, 2006.

\bibitem{BDV}
H.~Beir{\~a}o~da Veiga.
\newblock Perturbation theorems for linear hyperbolic mixed problems and
  applications to the compressible {E}uler equations.
\newblock {\em Comm. Pure Appl. Math.}, 46(2):221--259, 1993.

\bibitem{cor}
J.-M. Coron.
\newblock {On the controllability of 2-D incompressible perfect fluids}.
\newblock {\em J. Math. Pures Appl.}, 75(2):155--188, 1996.

\bibitem{corfurs}
J.-M. Coron and A.~V. Fursikov.
\newblock Global exact controllability of the {$2$}{D} {N}avier-{S}tokes
  equations on a manifold without boundary.
\newblock {\em Russian J. Math. Phys.}, 4(4):429--448, 1996.

\bibitem{fpgi}
E.~Fern\'{a}ndez-Cara, S.~Guerrero, O.~Yu. Imanuvilov, and J.~P.
Puel.
\newblock {Local exact controllability of the Navier--Stokes system}.
\newblock {\em J. Math. Pures Appl.}, 83(12):1501--1542, 2004.

\bibitem{fuim}
A.~V. Fursikov and O.~Yu. Imanuvilov.
\newblock {Exact controllability of the Navier--Stokes and Boussinesq
  equations}.
\newblock {\em Russian Math. Surveys}, 54(3):93--146, 1999.

\bibitem{gla}
O.~Glass.
\newblock {Exact boundary controllability of 3-D Euler equation}.
\newblock {\em ESAIM Control Optim. Calc. Var.}, 5:1--44, 2000.

\bibitem{Gla2}
O.~Glass.
\newblock On the controllability of the 1-{D} isentropic {E}uler equation.
\newblock {\em J. Eur. Math. Soc. (JEMS)}, 9(3):427--486, 2007.


\bibitem{ima2001}
 O.~Yu. Imanuvilov.
\newblock  {Remarks on exact controllability for the {N}avier-{S}tokes
              equations}.
\newblock {\em ESAIM Control Optim. Calc. Var.},  6:39--72, 2001.


\bibitem{Jur}
V.~Jurdjevic.
\newblock {\em Geometric control theory}, volume~52 of {\em Cambridge Studies
  in Advanced Mathematics}.
\newblock Cambridge University Press, Cambridge, 1997.

\bibitem{Kato1}
T.~Kato.
\newblock The {C}auchy problem for quasi-linear symmetric hyperbolic systems.
\newblock {\em Arch. Rational Mech. Anal.}, 58(3):181--205, 1975.

\bibitem{LiRao}
T.~Li and B.~Rao.
\newblock Exact boundary controllability for quasi-linear hyperbolic systems.
\newblock {\em SIAM J. Control Optim.}, 41(6):1748--1755, 2003.

\bibitem{hn}
H.~Nersisyan.
\newblock {Controllability of 3D incompressible Euler equations by a
  finite-dimensional external force}.
\newblock {\em ESAIM Control Optim. Calc. Var.}, 16(3):677-- 694, 2010.

\bibitem{rod}
S.~S. Rodrigues.
\newblock {Navier--Stokes equation on the rectangle: controllability by means
  of low mode forcing}.
\newblock {\em J. Dyn. Control Syst.}, 12(4):517--562, 2006.

\bibitem{shi1}
A.~Shirikyan.
\newblock {Approximate controllability of three-dimensional Navier- Stokes
  equations}.
\newblock {\em Comm. Math. Phys.}, 266(1):123--151, 2006.

\bibitem{shi2}
A.~Shirikyan.
\newblock {Exact controllability in projections for three-dimensional
  Navier-Stokes equations}.
\newblock {\em Annales de l'IHP, Analyse Non Lin\'eaire}, 24:521--537, 2007.

\bibitem{shi3}
A.~Shirikyan.
\newblock {Euler equations are not exactly controllable by a finite-dimensional
  external force}.
\newblock {\em Physica D}, 237:1317--1323, 2008.

\bibitem{tey1}
M.~E. Taylor.
\newblock {Partial Differential Equations, III}.
\newblock {\em Springer-Verlag, New York}, 1996.

\end{thebibliography}

\end{document}